\documentclass[reqno]{amsart}
\usepackage{hyperref}

\vspace{9mm}

\begin{document}

\title[\hfilneg \hfil Controllability]{Controllability of a second-order impulsive neutral differential equation via resolvent operator technique}
\author[A. Afreen,  A. Raheem  \& A. Khatoon \hfil \hfilneg]
{A. Afreen$^{*}$,   A. Raheem \& A. Khatoon}

\address{A. Afreen \newline
	Department  of Mathematics,
	Aligarh Muslim University,\newline Aligarh -
	202002, India.} \email{afreen.asma52@gmail.com}

\address{A. Raheem \newline Department of Mathematics,
	Aligarh Muslim University,\newline Aligarh -
	202002, India.} \email{araheem.iitk3239@gmail.com}

\address{A. Khatoon \newline
	Department  of Mathematics,
	Aligarh Muslim University,\newline Aligarh -
	202002, India.} \email{areefakhatoon@gmail.com}

\renewcommand{\thefootnote}{} \footnote{$^*$ Corresponding author:
	\url{ A. Afreen (afreen.asma52@gmail.com)}}

\subjclass[2010]{93B05, 34G20, 34K30, 34K40, 34K45}

\keywords{non-autonomous, neutral system, controllability, resolvent operator theory, semigroup theory}

\begin{abstract}
This paper uses the resolvent operator technique to investigate second-order non-autonomous neutral integrodifferential equations with impulsive conditions in a Banach space. We study the existence of a mild solution and the system's approximate controllability. The semigroup and resolvent operator theory,  graph norm, and Krasnoselskii's fixed point theorem are used to demonstrate the results. Finally, we present our findings with an example.
\end{abstract}

\maketitle \numberwithin{equation}{section}
\newtheorem{theorem}{Theorem}[section]
\newtheorem{lemma}[theorem]{Lemma}
\newtheorem{proposition}[theorem]{Proposition}
\newtheorem{corollary}[theorem]{Corollary}
\newtheorem{remark}[theorem]{Remark}
\newtheorem{definition}[theorem]{Definition}
\newtheorem{example}[theorem]{Example}
\allowdisplaybreaks

\section{\textbf{Introduction}} Let $\big(\mathfrak{B},\| \cdot\|\big)$ be a Banach space. Consider the control systems governed by the following neutral integrodifferential  equations with impulses:
\begin{eqnarray} \label{1.1}
	\left \{ \begin{array}{lll} \dfrac{d^2}{dt^2}E(t,\vartheta_t)=\mathcal{A}(t)E(t,\vartheta_t)+\displaystyle\int_{0}^{t}\zeta(t,s)E(s,\vartheta_s)ds+	\pounds_1\big(t,\vartheta(t)\big) +\beta u(t),	\\ \hspace{7.8cm} t \in T= [0,\ell],~ t\neq t_q,
		&\\\Delta \vartheta(t_q)=\mathrm{I}_q\big(\vartheta(t_q)\big), \quad q=1,2,\ldots,N,
		&\\\Delta \vartheta'(t_q)=\mathrm{J}_q\big(\vartheta(t_q)\big), \quad q=1,2,\ldots,N,
		&\\ \vartheta_0=\varPhi \in \wp, \quad \vartheta'(0)=x^1 \in \mathfrak{B}.  
		
	\end{array}\right.
\end{eqnarray}
 $\mathcal{A}(t): \mathfrak{D}\big(\mathcal{A}(t)\big)\subseteq \mathfrak{B}\rightarrow  \mathfrak{B}$ is a densely defined closed linear operator and $\vartheta _t:(-\infty,0]\rightarrow  \mathfrak{B},~ \vartheta_t(s)=\vartheta(t+s)$ belongs to an abstract phase space $\wp$. $\pounds_1:[0,\ell]\times \mathfrak{B}\rightarrow  \mathfrak{B};~\pounds_2, E:[0,\ell]\times \wp\rightarrow  \mathfrak{B},$ with $E(t,\varPsi)=\varPsi(0)+\pounds_2(t, \varPsi)$ are appropriate functions. $\zeta(t,s):\mathfrak{D}(\zeta)\subseteq \mathfrak{B}\rightarrow  \mathfrak{B}$ is a closed linear operator whose domain does not depend on $(t,s).$  Also, assume that $\frac{d}{dt}\pounds_2(t,\vartheta_t)\big\rvert_{t=0}=y^1.$ Let $0=t_0<t_1<t_2<\ldots<t_N<t_{N+1}=\ell;~ \mathrm{I}_q, \mathrm{J}_q:\mathfrak{B}\rightarrow  \mathfrak{B},~ q=1,2,\ldots,N$ are impulsive functions, $\Delta \vartheta(t_q)$ and $ \Delta \vartheta'(t_q)$ indicate the jump of $\vartheta$ and $\vartheta'$ at $t_q.$ Let $PC(T,\mathfrak{B})=\big\{\vartheta:T \rightarrow \mathfrak{B}~|~\vartheta(t)~ \mbox{is continuously differentiable at}~ t\neq t_q; \mbox{ and} ~\vartheta(t_q)=\vartheta(t_q^-), \vartheta'(t_q)=\vartheta'(t_q^-)\big\}.$ $u(\cdot)$ is the control function in a Banach space $ L^2(T,U),$ where $U$ is a Banach space; the operator $\beta:U\rightarrow  \mathfrak{B}$ is linear and continuous.

To investigate the above problem, the existence of a resolvent operator associated with the homogeneous system 
\begin{eqnarray*} 
	\vartheta''(t)=\mathcal{A}(t)\vartheta(t)+\int_{0}^{t}\zeta(t,s)\vartheta(s)ds, \quad t\in T=[0,\ell],
\end{eqnarray*}
is used. Researchers have expressed an interest in studying various problems using the resolvent operator technique in recent years. The resolvent operator takes the place of the $C_ 0$-semigroup in evolution equations and is crucial in solving differential equations in both the weak and strict sense \cite{5,17,3,1,k1,20}. Using the resolvent operators, Grimmer \cite{12} investigated the existence of mild solutions to evolution equations. Rezapour et al. \cite{10} noted the existence of mild solutions via the resolvent operator technique.

It is well known that many real-life phenomena are affected by the sudden change in their state at certain moments, such as heartbeats and blood flow in the human body. These phenomena are discussed in the form of impulses whose duration is negligible compared to the whole process. Such processes are modeled using impulsive differential equations. Analyzing mathematical models requires an understanding of impulsive systems. It has a wide range of applications, including drug diffusion in the human body, population dynamics, theoretical physics, mathematical economy, chemical technology, engineering, control theory, medicine, and so on. More information can be found in the references \cite{6,5,2,18,9,25,32,31}. 

The goal of controllability theory is the ability to control a specific system to the desired state. One of the fundamental ideas for investigating and analyzing various dynamical control processes is controllability theory. Many research papers on the controllability of second-order linear and nonlinear differential systems were presented, for example, see \cite{23,24,d4,29,26}. Most researchers have been found to study autonomous and non-autonomous systems using various techniques, see \cite{15,h3,d,4,8,29,11,16,14}.

Many real-world phenomena can be represented mathematically by a dynamical system governed by neutral differential equations with nonlocal conditions \cite{b1}. Nonlocal problems are more commonly used in applications than classical problems. Many researchers are now paying close attention to this theory and its applications. For more details, refer \cite{h3,7,5,30}.

There are five sections to this paper. The first two sections include an introduction, notations, some required definitions, assumptions, as well as some lemmas. The third section discusses the existence of a mild solution and the system's controllability. To demonstrate the results, an example is provided in the fourth section, and in the last section, a brief conclusion is given.

\section{\textbf{Preliminaries and Assumptions}}
The space $\mathfrak{D}(\mathcal{A})$ provided with the graph norm induced by $\mathcal{A}(t)$ is a Banach space. We will assume that all of these norms are equivalent. A simple condition for obtaining this property is that there exists $\lambda \in \rho(\mathcal{A}(t)),$ the resolvent set of $\mathcal{A}(t)$, so that $(\lambda I- \mathcal{A}(t))^{-1}:\mathfrak{B} \rightarrow \mathfrak{D}(\mathcal{A})$ is a bounded linear operator. In what follows, by $[\mathfrak{D}(\mathcal{A})]$ we represent the vector space $\mathfrak{D}(\mathcal{A})$ provided with any of these equivalent norms, and we denote
$$ \|\vartheta\|_{[\mathfrak{D}(\mathcal{A})]}=\|\vartheta\|+\|\mathcal{A}(t)\vartheta\|, ~\vartheta \in \mathfrak{D}(\mathcal{A}).$$

Let us assume that for every $\vartheta \in \mathfrak{D}(\mathcal{A}),~ t\longmapsto \mathcal{A}(t) \vartheta $ is continuous. Therefore, we consider that $\mathcal{A}(\cdot)$ generates $\{S(t,s)\}_{0\leq s \leq t \leq \ell},$ refer to \cite{29}. We regard that $S(\cdot)\vartheta$ is continuously differentiable for all $\vartheta \in \mathfrak{B}$ with derivative uniformly bounded on bounded intervals.
 We introduce another operator ${C(t,s)}$ associated with the evolution operator $S(t,s)$ as $$C(t,s)=-\frac{\partial S(t,s)}{\partial s}.$$

Let $\varOmega=\big\{(t,s):0\leq s \leq t \leq \ell\big\}.$ The following assumptions hold throughout the paper.\\
\begin{itemize}
	\item[(A1)] The operator $\zeta(\cdot,\cdot):[\mathfrak{D}(\mathcal{A})]\rightarrow \mathfrak{B} $ is bounded and continuous, i.e.,\\
	$$\|\zeta(t,s)\vartheta\| \leq \hbar_1 \|\vartheta\|_{[\mathfrak{D}(\mathcal{A})]},$$
	where $(t,s) \in \varOmega$ and $\hbar_1>0.$	
	\item[(A2)]  There exists $L_{\zeta}>0$ such that
	$$\|\zeta(t_2,s)\vartheta-\zeta(t_1,s)\vartheta\| \leq L_{\zeta}|t_2-t_1| \|\vartheta\|_{[\mathfrak{D}(\mathcal{A})]},$$
for all $\vartheta \in \mathfrak{D}(\mathcal{A}), ~ 0\leq s \leq t_1\leq t_2\leq \ell.$
		\item[(A3)] There exists $\hbar_2>0$ such that
		$$\bigg\|\int_{\eta}^{t} S(t,s)\zeta(s,\eta)\vartheta ds\bigg\| \leq \hbar_2 \|\vartheta\|,$$
		for all $\vartheta \in \mathfrak{D}(\mathcal{A}). $
\end{itemize}
 \begin{definition} A two-parameter family $\{\Re(t,s)\}_{t\geq s}$ on $\mathfrak{B}$ is said to be a resolvent operator for the system
\begin{eqnarray} \label{1.2}
 	\left \{ \begin{array}{ll}	\vartheta''(t)=\mathcal{A}(t)\vartheta(t)+\displaystyle\int_{s}^{t}\zeta(t,\tau)\vartheta(\tau)d\tau, \quad s\leq t \leq \ell,
 &\\\vartheta(s)=0, \quad \vartheta'(s)=x \in \mathfrak{B},	\end{array}\right.
\end{eqnarray}
if
\begin{itemize}
	\item[(i)] the map $\Re:\varOmega \rightarrow L(\mathfrak{B})$ is strongly continuous, $\Re(t,\cdot)\vartheta$ is continuously differentiable $\forall$ $\vartheta \in \mathfrak{B}, \Re(s,s)=0, \dfrac{\partial}{\partial t}\Re(t,s)\Big\rvert_{t=s}=\mathcal{I}$ and $\dfrac{\partial}{\partial s}\Re(t,s)\Big\rvert_{s=t}=-\mathcal{I};$
\item[(ii)] $\Re(t,\cdot)x$ is a solution for (\ref{1.2}), i.e.,
\begin{eqnarray*}\frac{\partial^2}{\partial t^2}\Re(t,s)x=\mathcal{A}(t)\Re(t,s)x+\int_{s}^{t}\zeta(t,\tau)\Re(\tau,s)xd\tau,
	\end{eqnarray*}
for all $0 \leq s\leq t \leq \ell,$ $x \in  \mathfrak{D}(\mathcal{A}).$
\end{itemize}
\end{definition}
We may assume that 
$$\|\Re(t,s)\|\leq M_1,~\Big\|\frac{\partial}{\partial s}\Re(t,s)\Big\|\leq M_2, \quad (t,s) \in \varOmega,$$
where $M_1>0$ and $M_2>0.$ 

Next, the linear operator
\begin{eqnarray*}
	\tilde{F}(t,\eta)x=\int_{\eta}^{t} \zeta(t,s)	\Re(s,\eta)x ds, \quad x \in  \mathfrak{D}(\mathcal{A}),~ 0 \leq \eta \leq t \leq \ell,
\end{eqnarray*}
can be extended to $\mathfrak{B}$ and still denoted by itself $ 	\tilde{F}(t,\eta),~  \tilde{F}:\varOmega \rightarrow L(\mathfrak{B}) $ is strongly continuous, and 
\begin{eqnarray}
	\Re(t,\eta)x=S(t,\eta)+\int_{\eta}^{t} S(t,s)\tilde{F}(s,\eta)x ds, \quad \mbox{for all}~ x \in \mathfrak{B}.
		\end{eqnarray}
Clearly, $	\Re(\cdot)$ is uniformly Lipschitz continuous, i.e., there exists  a constant $L_{\Re}>0$ such that
	\begin{eqnarray}\|	\Re(t+h,\eta)-	\Re(t,\eta)\|\leq L_{\Re}|h|,\quad \mbox{for all} ~t,t+h,\eta \in[0,\ell],
	\end{eqnarray}
and
	\begin{eqnarray}\bigg\|\frac{\partial \Re(t+h,\eta)}{\partial s}-\frac{\partial \Re(t,\eta)}{\partial s}\bigg\|\leq M_\Re|h|,
\end{eqnarray} for $M_\Re>0.$

The phase space $\wp$ has the following properties:
\begin{itemize}
	\item[(B1)] If $\vartheta:(-\infty, \mu+\ell) \rightarrow \mathfrak{B}, \ell>0, \mu \in \mathbb{R}$ is continuous on $[\mu,\mu+\ell)$ and $\vartheta_\mu \in \wp,$ then for every $t \in [\mu,\mu+\ell),$
we have\\
(i) $\vartheta_t$ is in $\wp;$\\
(ii) $\|\vartheta(t)\| \leq K_1 \|\vartheta_t\|_{\wp};$\\
(iii) $\|\vartheta_t\|_{\wp} \leq K_2(t-\mu)\sup\{\|\vartheta(s)\|:\mu \leq s \leq t\}+K_3(t-\mu)\|\vartheta_\mu\|_{\wp},$\\
where $K_1>0$ is a constant; $K_2, K_3:[0,\infty) \rightarrow [1, \infty), $ ~$K_2(\cdot)$ is continuous, $K_3(\cdot)$ is locally bounded, and $K_1, K_2, K_3$ are independent of $\vartheta(\cdot).$
\item[(B2)] For $\vartheta(\cdot)$ in (B1), $t \rightarrow \vartheta_t$ is continuous from $[\mu,\mu+\ell)$ into $\wp.$
\item[(B3)] $\wp$ is complete.
\end{itemize}
For each $r>0,$ define the set $\Theta_r=\big\{\vartheta\in \mathcal{K}(\ell): \|\vartheta(t)\|\leq r ~\mbox{for all} ~t \in T\big\},$ where $\mathcal{K}(\ell)=\big\{\vartheta\in PC(T,\mathfrak{B}):\vartheta(0)=\varPhi(0)\big\}$ is a convex closed subset of $PC(T,\mathfrak{B}).$ To establish the results, we will introduce the required assumptions as follows:
\begin{itemize}
	\item[(C1)] The function $\pounds_1:T\times \mathfrak{B} \rightarrow \mathfrak{B}$ is continuous for all $t \in T$ and strongly measurable for each $\vartheta \in \mathfrak{B}.$ In addition, there is $\nu_r(\cdot)\in L^2(T,\mathbb{R^+})$ for each $r>0$ such that
	\begin{eqnarray*}
		\sup\big\{\|\pounds_1(t,\vartheta(t))\|:\|\vartheta(t)\|\leq r\big\}\leq \nu_r(t), \quad\mbox{for a.e.}~ t \in T,
			\end{eqnarray*}
		and \begin{eqnarray*}
			\lim_{r\rightarrow \infty}\inf \frac{1}{r}\|\nu_r\|_{L^2}=\sigma<\infty.
	 	\end{eqnarray*}
 	\item[(C2)] The closure of $\big\{\Re(t,s)\big[\pounds_1(s,\vartheta(s)) +\beta u(s)\big]:(t,s)\in \Omega,\|\vartheta(s)\|\leq r\big\}$ is compact in $\mathfrak{B}.$
 	\item[(C3)] The function $\pounds_2:T\times \wp \rightarrow \mathfrak{B}$ is continuous and satisfying:\\
 	(i) The set $\big\{\pounds_2(t,\varPsi):\|\varPsi\|_{\wp}\leq \delta\big\}$ is equicontinuous on $T$ and is relatively compact in $\mathfrak{B},$ where $\delta>0.$ \\
 	(ii) There exist $r_1, r_2>0$ such that\\
 	$$\|\pounds_2(t,\varPsi)\|\leq r_1+r_2\|\varPsi\|_{\wp},$$
 	$\forall$ $t \in T$ and $\varPsi \in \wp.$\\
 (iii) There exists $L_2>0$ such that
 $$ 	\big\|\pounds_2(t,\varPsi^1)-\pounds_2(s,\varPsi^2)\big\|\leq L_2\big\|\varPsi^1-\varPsi^2\big\|_{\wp},$$
 $\forall$ $t, s\in T;$ $\varPsi^1, \varPsi^2 \in \wp.$
 \item[(C4)] The function $ \mathrm{I}_q, \mathrm{J}_q:\mathfrak{B}\rightarrow \mathfrak{B}, \quad q=1,2,\ldots,N,$ are continuous operators and there are $d_q, e_q>0$ such that
 $$\big\|\mathrm{I}_q(\vartheta)\big\|\leq d_q\big(\|\vartheta\|+1\big), $$
  $$\big\|\mathrm{J}_q(\vartheta)\big\|\leq e_q\big(\|\vartheta\|+1\big), $$
  where $q=1,2,\ldots,N,~ \vartheta \in \mathfrak{B}. $
  \item[(C5)] The control $u(\cdot)=u(t,\vartheta) \in  L^2(T,U)$ is continuous on $T$ and bounded for every $\vartheta \in \mathfrak{B},$ i.e., $\exists$ $\lambda>0$ such that $\|u(t)\|=\|u(t,\vartheta)\|\leq \lambda \|\vartheta\|.$
\end{itemize}
\begin{lemma}\label{2.2}\cite{5}
$\Re(t,s)$ is compact and continuous in the uniform operator topology of $L(\mathfrak{B})$ for $t,s>0.$ 
	\end{lemma}
\begin{definition}
	A function $\vartheta:(-\infty,\ell] \rightarrow \mathfrak{B}$ is called a mild solution for (\ref{1.1}) if $\vartheta$ is continuous on $[0,\ell],~ \vartheta_0=\varPhi \in \wp,~ \frac{d}{dt}E(t,\vartheta_t)\big\rvert_{t=0}=x^1+y^1$ and 
	\begin{eqnarray*}
		\vartheta(t)&=&-{\frac{\partial \Re(t,s)\big[\varPhi(0)+\pounds_2(0,\varPhi)\big]}{\partial s}}\bigg\rvert_{s=0}+\Re(t,0)\big[x^1+y^1\big]-\pounds_2(t,\vartheta_t)\\&& \quad +\int_{0}^{t}\Re(t,s)\big[\pounds_1\big(s,\vartheta(s)\big)+\beta u(s)\big] ds\\&& \quad-\sum_{0<t_q<t} \frac{\partial\Re(t,s)}{\partial s}\bigg\rvert_{s=t_q}\mathrm{I}_q\big(\vartheta(t_q)\big) + \sum_{0<t_q<t}\Re(t,t_q)\mathrm{J}_q\big(z(t_q)\big)
	\end{eqnarray*}
is satisfied for $t \in T.$
\end{definition}
\begin{definition}
 The  system (\ref{1.1}) is approximately controllable on $T,$ if for each final state $x_\ell \in \mathfrak{B} $ and each $\varepsilon>0,$  there exists a control $u(\cdot) \in   L^2(T,U)$ such that the mild solution $\vartheta(\cdot,u)$ of (\ref{1.1}) satisfies $\|\vartheta(\ell,u)-x_\ell\|<\varepsilon.$ 
\end{definition}
To show the controllability of (\ref{1.1}), we consider the linear system:
\begin{eqnarray} \label{1.3}
	\left \{ \begin{array}{ll} \dfrac{d^2}{dt^2}E(t,\vartheta_t)=\mathcal{A}(t)E(t,\vartheta_t)+\displaystyle\int_{0}^{t}\zeta(t,s)E(s,\vartheta_s)ds+\beta u(t),	\\\hspace{7cm} t \in T= [0,\ell],
		&\\ \vartheta_0=\varPhi \in \wp, \quad \vartheta'(0)=x^1 \in \mathfrak{B}.  
		
	\end{array}\right.
\end{eqnarray}
Define the controllability operator $\Gamma_{0}^{\ell}: \mathfrak{B} \rightarrow  \mathfrak{B}$ as 
\begin{eqnarray*}
	\Gamma_{0}^{\ell}=\int_{0}^{\ell}\Re(\ell,s)\beta \beta^*\Re^*(\ell,s)ds,
\end{eqnarray*}
where $*$ denotes the adjoint and the resolvent of $\Gamma_{0}^{\ell}$ is given by
\begin{eqnarray*}
\mathcal{V}\big(\varepsilon,\Gamma_{0}^{\ell}\big)=\big(\varepsilon I+\Gamma_{0}^{\ell}\big)^{-1}, \quad \varepsilon>0.
\end{eqnarray*}

\begin{lemma}\label{2.1}\cite{7}
	The linear control system (\ref{1.3}) is approximately controllable on $T$ if and only if $\varepsilon \mathcal{V}\big(\varepsilon,\Gamma_{0}^{\ell}\big)  \rightarrow  0$ as $  \varepsilon\rightarrow 0^+ $ in the strong operator topology.
	\end{lemma}
\section{\textbf{Main results}}
\begin{lemma}\label{3.1}
	If the conditions (C1),(C2),(C3) and (C5) hold and $u(\cdot) \in  L^2(T,U)$ is bounded, then the operator $P_1:\Theta_r\rightarrow \Theta_r,$ defined by 
	\begin{eqnarray}\label{1}
		(P_1\vartheta)(t)&=&-{\frac{\partial \Re(t,s)\big[\varPhi(0)+\pounds_2(0,\varPhi)\big]}{\partial s}}\bigg\rvert_{s=0}-\pounds_2(t,\vartheta_t)\nonumber\\&& \quad +\int_{0}^{t}\Re(t,s)\big[\pounds_1\big(s,\vartheta(s)\big)+\beta u(s)\big] ds
	\end{eqnarray}
is compact.
	\end{lemma}
\begin{proof}
	First, we show that $P_1(\Theta_r)$ is equicontinuous on $[0,\ell].$ For any $s_1,s_2 \in[0,\ell]$ with $s_1<s_2$ and $\vartheta \in \Theta_r,$ we have\\
$\big\|(P_1\vartheta)(s_2)-(P_1\vartheta)(s_1)\big\|$
	\begin{eqnarray*}
	&\leq& \bigg\|\bigg({\frac{\partial \Re(s_1,s)}{\partial s}}-{\frac{\partial \Re(s_2,s)}{\partial s}}\bigg)\bigg\rvert_{s=0}\big[\varPhi(0)+\pounds_2(0,\varPhi)\big]\bigg\|\\&&
		\quad +\Big\|\pounds_2\big(s_1,\vartheta_{s_1}\big)-\pounds_2\big(s_2,\vartheta_{s_2}\big)\Big\|\\&&
		\quad+ \bigg\|\int_{0}^{s_1}\big[\Re(s_2,s)-\Re(s_1,s)\big]\big[\pounds_1\big(s,\vartheta(s)\big)+\beta u(s)\big] ds\bigg\|\\&&
		\quad +\bigg\|\int_{s_1}^{s_2}\Re(s_2,s)\big[\pounds_1\big(s,\vartheta(s)\big)+\beta u(s)\big] ds\bigg\|\\
		&\leq & M_\Re |s_2-s_1|\big[r+\big(r_1+r_2\|\varPhi\|_\wp\big)\big]+L_2\|\vartheta_{s_1}-\vartheta_{s_2}\|_\wp\\&&\quad +\int_{0}^{s_1}L_\Re|s_2-s_1|\big[\|\pounds_1\big(s,\vartheta(s)\big)\|+\|\beta u(s)\|\big]ds\\&&
		\quad+M_1\int_{s_1}^{s_2}\big[\|\pounds_1\big(s,\vartheta(s)\big)\|+\|\beta u(s)\|\big]ds.
		\end{eqnarray*}
Therefore, $	\big\|(P_1\vartheta)(s_2)-(P_1\vartheta)(s_1)\big\| \rightarrow 0$	as $s_2 \rightarrow s_1.$
Hence, $P_1(\Theta _r)$ is equicontinuous on $[0,\ell]$.
Following the idea used in \cite{10}, we can easily show that $P_1(\Theta_r)$ is relatively compact. Hence, by Arzela-Ascoli theorem, the operator $P_1$ is compact.
	\end{proof}
\begin{theorem}\label{3.2}
	If the conditions (B1)-(B3) and (C1)-(C5) hold, then the system (\ref{1.1}) has a mild solution in some $\Theta _r,$ if
	\begin{eqnarray}
		2M_1+M_2+r_2K_2(t-\mu)+M_1\ell\sigma+M_1\|\beta\|\lambda \ell+M_1\sum\limits_{q=0}^{N}e_q+M_2\sum\limits_{q=0}^{N}d_q<1.
	\end{eqnarray}
		\end{theorem}
	\begin{proof}
Take the control $u(\cdot)$ as $u(t)=u\big(t,\vartheta(\ell)\big), ~\vartheta \in\Theta _r. $ Define $\tilde{Q}:\Theta _r \rightarrow PC(T,\mathfrak{B})$ by
 	\begin{eqnarray}
 	\big(\tilde{Q}\vartheta\big)(t)=\Re(t,0)\big[x^1+y^1\big]+(P_1\vartheta)(t)+(P_2\vartheta)(t),
 \end{eqnarray} where the map $P_1$ is defined by (\ref{1}) and 
\begin{eqnarray}
(P_2\vartheta)(t)=\sum_{0<t_q<t}\Re(t,t_q)\mathrm{J}_q\big(\vartheta(t_q)\big)-\sum_{0<t_q<t} \frac{\partial\Re(t,s)}{\partial s}\bigg\rvert_{s=t_q}\mathrm{I}_q\big(\vartheta(t_q)\big). \end{eqnarray}
We show that the map $\tilde{Q}$ maps from $\Theta _r$ to $\Theta _r$ for some $r>0.$ Let us assume contrarily that for each $r>0,$ there exist $\vartheta_r \in \Theta _r$ such that $\big\|\tilde{Q}\vartheta_r(t)\big\|>r$ for some $t\in T$. Therefore,
\begin{eqnarray}\label{2}
	r \leq \big\|(\tilde{Q}\vartheta_r)(t)\big\|&\leq& M_1\big\|x^1+y^1\big\|+M_2\big[\|\varPhi(0)\|+\|\pounds_2(0,\varPhi)\|\big]+\big\|\pounds_2(t,\vartheta_t)\big\|\nonumber \\&& \quad+M_1\int_{0}^{t}\big[\|\pounds_1\big(s,\vartheta(s)\big)\|+\|\beta u(s)\|\big] ds +M_1\sum\limits_{q=0}^{N}\big\|\mathrm{J}_q\big(\vartheta(t_q)\big)\big\|\nonumber \\&& \quad+M_2\sum\limits_{q=0}^{N}\big\|\mathrm{I}_q\big(\vartheta(t_q)\big)\big\|\nonumber \\
	&\leq& 2rM_1+M_2\big[r+\big(r_1+r_2\|\varPhi\|_\wp\big)\big]+\big(r_1+r_2\|\vartheta_t\|_\wp\big)+M_1\ell\|\nu_r\|_{L^2}\nonumber \\&& \quad+M_1\|\beta \|\lambda \ell r+M_1\sum_{q=0}^{N}e_q(r+1)+M_2\sum_{q=0}^{N}d_q(r+1)\nonumber \\
	&\leq& 2rM_1+M_2\big[r+\big(r_1+r_2\|\varPhi\|_\wp\big)\big]+r_1+r_2rK_2(t-\mu)\nonumber \\&& \quad+r_2K_3(t-\mu)\|\vartheta_\mu\|_{\wp}+M_1\ell \|\nu_r\|_{L^2}+M_1\|\beta\|\lambda \ell r\nonumber \\&& \quad+M_1\sum\limits_{q=0}^{N}e_q(r+1)+M_2\sum\limits_{q=0}^{N}d_q(r+1).
	\end{eqnarray}
Dividing both sides of (\ref{2}) by $r$ and then taking $r \rightarrow \infty ,$ we get
\begin{eqnarray*}
	1<2M_1+M_2+r_2K_2(t-\mu)+M_1 \ell\sigma+M_1\|\beta \|\lambda \ell+M_1\sum\limits_{q=0}^{N}e_q+M_2\sum\limits_{q=0}^{N}d_q.
\end{eqnarray*}
This is a contradiction. So, there exists  $r>0$ such that $\tilde{Q}(\Theta_r) \subseteq \Theta_r.$ Further, we prove that $\tilde{Q}$ is continuous on $\Theta_r.$ Take any sequence $\{\vartheta_n\}\subset \Theta_r$ with $\vartheta_n \rightarrow \vartheta \in \Theta_r$ as $n \rightarrow \infty.$ Then, for each $t \in T,$ we have
\begin{itemize}
	\item [(a)] $\big\|\pounds_1\big(t,\vartheta_n(t)\big)-\pounds_1\big(t,\vartheta(t)\big)\big\| \rightarrow 0$ as $n \rightarrow \infty.$
	\item[(b)] $\big\|\pounds_1\big(t,\vartheta_n(t)\big)-\pounds_1\big(t,\vartheta(t)\big)\big\| \leq 2 \nu_r(t).$
	\item[(c)]  $\big\|\beta u\big(t,\vartheta_n(t)\big)-\beta u\big(t,\vartheta(t)\big)\big\|<\infty.$
\end{itemize}

\noindent Using Lebesgue Dominated Convergence theorem, we have\\

$	\big\|(\tilde{Q}\vartheta_n)(t)-(\tilde{Q}\vartheta)(t)\big\|$
\begin{eqnarray*}
&\leq& \big\|\pounds_2\big(t,(\vartheta_n)_t\big)-\pounds_2\big(t,\vartheta_t\big)\big\|+\bigg\|\int_{0}^{t}\Re(t,s)\big[\pounds_1\big(s,\vartheta_n(s)\big)-\pounds_1\big(s,\vartheta(s)\big)\big]ds\bigg\|\\&& \quad +\bigg\|\int_{0}^{t}\Re(t,s)\big[\beta u(s,\vartheta_n(s))-\beta u(s,\vartheta(s))\big]ds\bigg\|\\&& \quad +\sum\limits_{0<t_q<t}\Big\|\Re(t,t_q)\big[\mathrm{J}_q\big(\vartheta_n(t_q)\big)-\mathrm{J}_q\big(\vartheta(t_q)\big)\big]\Big\|\\&&\quad +\sum\limits_{0<t_q<t}\bigg\| \frac{\partial\Re(t,s)}{\partial s}\bigg\rvert_{s=t_q}\big[\mathrm{I}_q\big(\vartheta_n(t_q)\big)-\mathrm{I}_q\big(\vartheta(t_q)\big)\big]\bigg\|\\
	&\leq& L_2\big\|(\vartheta_n)_t-\vartheta_t\big\|_\wp+M_1\int_{0}^{\ell}\big\|\pounds_1\big(s,\vartheta_n(s)\big)-\pounds_1\big(s,\vartheta(s)\big)\big\|ds\\&&\quad +M_1\int_{0}^{\ell}\big\|\beta u\big(s,\vartheta_n(s)\big)-\beta u\big(s,\vartheta(s)\big)\big\|ds
	+M_1\sum\limits_{q=0}^{N}\Big\|\mathrm{J}_q\big(\vartheta_n(t_q)\big)-\mathrm{J}_q\big(\vartheta(t_q)\big)\Big\|\\&&\quad 
	+M_2\sum\limits_{q=0}^{N}\Big\|\mathrm{I}_q\big(\vartheta_n(t_q)\big)-\mathrm{I}_q\big(\vartheta(t_q)\big)\Big\|,\\&&	
	\rightarrow 0 ~\mbox{as}~ n \rightarrow \infty.
\end{eqnarray*}
Thus, $\tilde{Q}$ is continuous on $\Theta_r.$ Let $T_0=[0,t_1], T_1=(t_1,t_2],\ldots,T_N=(t_N,\ell],$ then the map $P_2$ can be rewritten as
\begin{eqnarray*} 
(P_2\vartheta)(t)=	\left \{ \begin{array}{lll}
		0,\quad t \in T_0,
		&\\ \Re(t,t_1)\mathrm{J}_1\big(\vartheta(t_1)\big)-\dfrac{\partial\Re(t,s)}{\partial s}\bigg\rvert_{s=t_1}\mathrm{I}_1\big(\vartheta(t_1)\big), \quad t \in T_1,
		&\\ \ldots,
		&\\\sum\limits_{q=1}^{N}\Re(t,t_q)\mathrm{J}_q\big(\vartheta(t_q)\big)-\sum\limits_{q=1}^{N} \dfrac{\partial\Re(t,s)}{\partial s}\bigg\rvert_{s=t_q}\mathrm{I}_q\big(\vartheta(t_q)\big), \quad t\in T_N. 
\end{array}\right.
\end{eqnarray*}
Since $\mathrm{I_1},\mathrm{J_1}$ are continuous and $\Re$ is compact. As a result, the set $\{P_2\vartheta:\vartheta\in \Theta_r\}$ is relatively compact in $\mathfrak{B}$ for every $t \in \overline{T_1}$  (closure of $T_1$). For any $s_1,s_2 \in \overline{T_1}$ with $s_1<s_2$ and $\vartheta \in \Theta_r,$ it follows from Lemma~\ref{2.2} that 
\begin{eqnarray*}
	\bigg\|\mathrm{J}_1\big(\vartheta(t_1)\big)\big[\Re(s_2,t_1)-\Re(s_1,t_1)\big]\bigg\|+\bigg\|\mathrm{I}_1\big(\vartheta(t_1)\big) \bigg(\frac{\partial\Re(s_2,s)}{\partial s}\bigg\rvert_{s=t_1}-\frac{\partial\Re(s_1,s)}{\partial s}\bigg\rvert_{s=t_1}\bigg)\bigg\|\\
	\leq e_1(r+1)L_\Re|s_2-s_1|+d_1(r+1)M_\Re|s_2-s_1|,\\
	\rightarrow 0 ~\mbox{as}~ s_1 \rightarrow s_2 ~\mbox{independent of }~ \vartheta\in \Theta_r.
\end{eqnarray*} 
Thus, the set $\big\{P_2\vartheta:\vartheta\in \Theta_r\big\}$ is equicontinuous on $\overline{T_1}.$ Therefore, by Arzela-Ascoli theorem, the set  $\big\{P_2\vartheta:\vartheta\in \Theta_r\big\}$ is relatively compact in $C(\overline{T_1},\mathfrak{B}).$ In the same manner, we can show for $T_q, q=2,3,\ldots,N.$ It follows from \cite{5} that $\big\{P_2\vartheta:\vartheta\in \Theta_r\big\}$ is relatively compact in $PC(T,\mathfrak{B}).$ Moreover, the map $\Re(t,0)\big[x^1+y^1\big]:\Theta_r \rightarrow \Theta_r$ is contraction as $\Re$ is Lipschitz continuous. But the map $P_1+P_2:\Theta_r \rightarrow \Theta_r$ is completely continuous. Therefore, by Krasnoselskii's fixed point theorem, $\tilde{Q}$ has a fixed point in $\Theta_r,$ which is the mild solution of (\ref{1.1}).
\end{proof}
\begin{theorem}\label{3.3}
	Suppose that hypotheses (B1)-(B3), (C1)-(C4) hold, and the functions $\pounds_1, \pounds_2,\mathrm{I}_q, \mathrm{J}_q,( q=1,2,\ldots,N)$ are uniformly bounded. If the associated linear system (\ref{1.3}) is approximately controllable on $T,$ then (\ref{1.1}) is approximately controllable on $T.$
\end{theorem}
\begin{proof}
	Let $b \in \mathfrak{B}$ and $\varepsilon>0$ be constant. Define the control $u(t)$ as
	\begin{eqnarray}
		u(t):=u_{\varepsilon}(t,\vartheta)=\beta ^* \Re^*(\ell,t)\mathcal{V}\big(\varepsilon,\Gamma_{0}^{\ell}\big)p(\vartheta),
	\end{eqnarray}
where \begin{eqnarray*}
	p(\vartheta)&=&b+{\frac{\partial \Re(\ell,s)\big[\varPhi(0)+\pounds_2(0,\varPhi)\big]}{\partial s}}\bigg\rvert_{s=0}-\Re(\ell,0)\big[x^1+y^1\big]+\pounds_2(\ell,\vartheta_\ell)\\&& \quad -\int_{0}^{\ell}\Re(\ell,s)\pounds_1\big(s,\vartheta(s)\big) ds+\sum\limits_{q=1}^{N} \frac{\partial\Re(\ell,s)}{\partial s}\bigg\rvert_{s=t_q}\mathrm{I}_q\big(\vartheta(t_q)\big)\\&& \quad - \sum\limits_{q=1}^{N}\Re(\ell,t_q)\mathrm{J}_q\big(\vartheta(t_q)\big).
\end{eqnarray*}
It follows from Lemma~\ref{3.1} and Theorem~\ref{3.2} that the control $u(t)$ satisfies (C5). Define a map $\tilde{Q}_\varepsilon:PC(T,\mathfrak{B}) \rightarrow PC(T,\mathfrak{B}) $ as defined in Theorem~\ref{3.2}
\begin{eqnarray}\label{7}
	(\tilde{Q}_\varepsilon \vartheta)(t)&=&\Re(t,0)\big[x^1+y^1\big]+(P_1\vartheta)(t)+\sum\limits_{q=1}^{N}\Re(t,t_q)\mathrm{J}_q\big(\vartheta(t_q)\big) \nonumber\\&& \quad-\sum\limits_{q=1}^{N} \frac{\partial\Re(t,s)}{\partial s}\bigg\rvert_{s=t_q}\mathrm{I}_q\big(\vartheta(t_q)\big),
\end{eqnarray}
where $P_1$ is given by (\ref{1}).
From Theorem~\ref{3.2}, it is clear that (\ref{7}) has a fixed point $\vartheta_\varepsilon,$ which is a mild solution of (\ref{1.1}). Let $\vartheta_\varepsilon \in \tilde{Q}_\varepsilon$ be a fixed point. Easily, we see that
\begin{eqnarray}\label{4}
	\vartheta_\varepsilon(\ell)&=&\Re(\ell,0)\big[x^1+y^1\big]+(P_1\vartheta_\varepsilon)(\ell)+\sum\limits_{q=1}^{N}\Re(\ell,t_q)\mathrm{J}_q\big(\vartheta_\varepsilon(t_q)\big) \nonumber\\&& \quad-\sum\limits_{q=1}^{N} \frac{\partial\Re(\ell,s)}{\partial s}\bigg\rvert_{s=t_q}\mathrm{I}_q\big(\vartheta_\varepsilon(t_q)\big)\nonumber \\&=& b-\varepsilon \mathcal{V}\big(\varepsilon,\Gamma_{0}^{\ell}\big)p(\vartheta_\varepsilon).
\end{eqnarray}
The uniform boundedness of $ \pounds_1, \pounds_2, \mathrm{I}_q,$ $ \mathrm{J}_q$ and compactness of $\Re(t,s)$ imply that there are subsequences of $\int_{0}^{\ell}\Re(\ell,s)\pounds_1(s,\vartheta_\varepsilon(s)) ds,$  $\pounds_2(\ell,(\vartheta_\varepsilon)_\ell),$ $\sum \limits_{q=1}^{N}\Re(\ell,t_q)\mathrm{J}_q\big(\vartheta_\varepsilon(t_q)\big)$ and $\sum \limits_{q=1}^{N} \frac{\partial\Re(\ell,s)}{\partial s}\Big\rvert_{s=t_q}\mathrm{I}_q\big(\vartheta_\varepsilon(t_q)\big)$   denoted by themselves, respectively; that converge to  $\hat{\pounds_1}, \hat{\pounds_2}, \hat{J}$ and $\hat{I},$  respectively. 
\\Let 
$\chi=b+{\dfrac{\partial \Re(\ell,s)\big[\varPhi(0)+\pounds_2(0,\varPhi)\big]}{\partial s}}\bigg\rvert_{s=0}-\Re(\ell,0)\big[x^1+y^1\big]+\hat{\pounds}_{2}-\hat{\pounds}_1+\hat{I}-\hat{J},$ then
\begin{eqnarray}\label{5}
	\big\|p(\vartheta_\varepsilon)-\chi\big\| &\leq& \big\|\pounds_2\big(\ell,(\vartheta_\varepsilon)_\ell\big)-\hat{\pounds_2}\big\|+\bigg\|\int_{0}^{\ell}\Re(\ell,s)\pounds_1\big(s,\vartheta_\varepsilon(s)\big) ds-\hat{\pounds_1}\bigg\|\nonumber\\&& \quad +\Bigg\|\sum\limits_{q=1}^{N} \frac{\partial\Re(\ell,s)}{\partial s}\bigg\rvert_{s=t_q}\mathrm{I}_q\big(\vartheta_\varepsilon(t_q)\big)-\hat{I}\Bigg\|\nonumber\\&& \quad +\Bigg\|\sum\limits_{q=1}^{N}\Re(\ell,t_q)\mathrm{J}_q\big(\vartheta_\varepsilon(t_q)\big)-\hat{J}\Bigg\|,\nonumber\\&&
	\rightarrow 0 ~\mbox{as} ~ \varepsilon \rightarrow 0.
	\end{eqnarray}
Using (\ref{4}), (\ref{5}) and Lemma~\ref{2.1}, we obtain
\begin{eqnarray*}
	\big\|\vartheta_\varepsilon(\ell)-b\big\|&\leq& \big\|\varepsilon \mathcal{V}\big(\varepsilon,\Gamma_{0}^{\ell}\big)(\chi)\big\|+\big\|\varepsilon \mathcal{V}\big(\varepsilon,\Gamma_{0}^{\ell}\big)\big\| \big\|p(\vartheta_\varepsilon)-\chi\big\|\\&& \rightarrow 0 ~\mbox{as}~ \varepsilon \rightarrow 0.
\end{eqnarray*}
	\end{proof}

\section{\textbf{Application}}
{\textbf{Example 4.1.}}
Consider the following neutral integro-differential equations:
\begin{eqnarray}  \label{4.1}
	\left \{ \begin{array}{lll} \dfrac{\partial^2}{\partial t^2}\big[\varrho(t,y)+\pounds_2\big(t,\varrho(t-\delta_1,y)\big)\big]\\\hspace{0.6cm} = \dfrac{\partial^2 }{\partial y^2}\big[\varrho(t,y)+\pounds_2\big(t,\varrho(t-\delta_1,y)\big)\big] +\mathcal{F}(t)\big[\varrho(t,y)+\pounds_2\big(t,\varrho(t-\delta_1,y)\big)\big] \\\hspace{0.8cm}+\displaystyle \int_{0}^{t}\hbar(t-s)\frac{\partial^2 }{\partial y^2}\big[\varrho(t,y)+\pounds_2\big(t,\varrho(t-\delta_1,y)\big)\big]ds\\\hspace{1cm}
		+\pounds_1\big(t,\varrho(t,y)\big)+u(t,y),
	\quad y \in [0,2\pi], ~t \in [0,\ell],~t \neq t_q,
		&\\\varrho(t,0)=\varrho(t,2\pi)=0,\quad t \in [0,\ell],  
		&\\\varrho(\theta,y)=\varPhi(\theta,y),  {\dfrac{\partial}{\partial t}}\varrho(t,y)\Big\rvert_{t=0}=b_1(y), \quad \theta \in(-\infty,0],~y \in [0, 2\pi],
		&\\\Delta \varrho(t,y)\big\rvert_{t=t_q}= \displaystyle\int_{0}^{2\pi}\varphi_q(\xi,y)\frac{(\varrho(t_q,\xi))^2}{\pi\big(1+(\varrho(t_q,\xi))^2\big)}d\xi, \quad q=1,2,3,\ldots,N,
		&\\\Delta \frac{\partial}{\partial t} \varrho(t,y)\big\rvert_{t=t_q}= \displaystyle\int_{0}^{2\pi}\varsigma_q(\xi,y)\frac{(\varrho(t_q,\xi))^4}{2e^2\big(1+(\varrho(t_q,\xi))^4\big)}d\xi,  \quad q=1,2,3,\ldots,N,
	\end{array}\right.
\end{eqnarray}

where  $t_q$ are the impulse points on $[0,\ell]$ such that $0=t_0<t_1<t_2<\ldots<t_N<t_{N+1}=\ell$ ($q=0, 1, 2,\ldots,N$), $\mathcal{F}, \hbar:[0,\ell]\rightarrow \mathbb{R},$ $\varPhi:(-\infty,0]\times [0,2\pi] \rightarrow \mathbb{R} $ and $b_1:[0,2\pi] \rightarrow \mathbb{R} $ satisfy suitable constraints.

Let $\mathfrak{B}=L^2(0,2\pi)$ be the Banach space endowed with norm $\|\cdot\|$ such that $\varPhi(\theta,\cdot), b_1(\cdot) \in \mathfrak{B}.$ Here, $\wp$ be a phase space and
$\varPhi(\theta)(y)=\varPhi(\theta,y)$ for  $\theta \in(-\infty,0],y \in [0, 2\pi].$ We will assume that the map $\pounds_1$ satisfies the conditions (C1)-(C2).

We define the operator $\mathcal{F}_0$ as $$\mathcal{F}_0\big[\varrho(t,y)+\pounds_2\big(t,\varrho(t-\delta_1,y)\big)\big]=\frac{\partial^2}{\partial y^2}\big[\varrho(t,y)+\pounds_2\big(t,\varrho(t-\delta_1,y)\big)\big] $$ with domain
$$\mathfrak{D}(A)=\big\{\varrho \in \mathfrak{B}|\varrho,\varrho'~ \mbox{are absolutely continuous}~ \varrho'' \in \mathfrak{B}, \mbox{and}~\varrho(0)=\varrho(2\pi)=0\big\}.$$ 
Clearly, $\mathcal{F}_0$ generates a cosine family $\{C(t)\}_{t\geq 0}$ on $\mathfrak{B}$ associated with sine family  $\{S(t)\}_{t\geq 0}.$ Moreover, it has discrete spectrum with eigen values $-m^2$ for $m\in \mathbb{N},$ whose corresponding eigen vectors are $$\vartheta_m(x)=\frac{1}{\sqrt{2\pi}}e^{imx}, ~m\in \mathbb{N}.$$
The set $\{\vartheta_m:m\in\mathbb{N}\}$ is an orthonormal basis of $\mathfrak{B}.$ Therefore, we have 
$$\mathcal{F}_0\varrho=\sum \limits_{m \in \mathbb{N}}{-m^2} \langle \varrho,\vartheta_m \rangle \vartheta_m, ~ \varrho\in \mathfrak{D}(\mathcal{F}_0),$$ 
$$C(t)\varrho=\sum \limits_{m \in \mathbb{N}}\cos(mt)\langle \varrho,\vartheta_m \rangle \vartheta_m, ~ t \in\mathbb{R},$$
and $$S(t)\varrho=\sum \limits_{m \in \mathbb{N}}\frac{\sin(mt)}{m}\langle \varrho,\vartheta_m \rangle \vartheta_m, ~ t \in\mathbb{R}.$$
Define $\mathcal{A}(t)\varrho=\mathcal{F}_0\varrho+\mathcal{F}(t)\varrho$ on $\mathfrak{D}(\mathcal{A}).$ By defining $\zeta(t,s)=\hbar(t-s)\mathcal{F}_0$ for $0\leq s\leq t \leq \ell$ on $\mathfrak{D}(\mathcal{A})$ and following the above, we can rewrite (\ref{4.1}) into an abstract form (\ref{1.1}). Furthermore, there exists a resolvent operator $\{\Re(t,s)\}_{0\leq s\leq t \leq \ell}$ associated with the homogeneous system of (\ref{4.1}). Also, we can easily show that all the assumptions are satisfied.\\

\textbf{Remark(i)} Under the above conditions, Theorem~\ref{3.2} guarantees that $\exists$ a mild solution for  (\ref{4.1}).\\ 
\textbf{(ii)} In addition, if the functions $\pounds_1, \pounds_2, \varphi_q, \varsigma_q ( q=1,2,\ldots,N)$ are uniformly bounded and the associated linear system of (\ref{4.1}) is approximately controllable on $[0,\ell],$ then by Theorem~\ref{3.3}, the system (\ref{4.1}) is approximately controllable on $[0,\ell].$

\section{ \textbf{Conclusion}} In this work, we established some sufficient conditions for the existence of mild solutions of the second-order non-autonomous neutral integrodifferential equations with impulsive conditions. And then, we study the approximate controllability of the system. The resolvent operator technique was used to obtain the results. We gave an example to show the effectuality of the results. By applying the same technique to fractional-order neutral differential equations, we will make some remarks on the nature of the system.

\textbf{{Acknowledgment}}
	{\it The authors would like to thank the referees for their valuable suggestions. The first and third authors acknowledge UGC, India, for providing financial support through MANF F.82-27/2019 (SA-III)/ 4453 and F.82-27/2019 (SA-III)/191620066959, respectively.}

\end{document}